\documentclass[12pt]{amsart}
\usepackage{setspace, amsmath, amsthm, amssymb, amsfonts, amscd, epic, graphicx, ulem, dsfont}
\usepackage[T1]{fontenc}
\usepackage{multirow}
\usepackage{bbm}
\usepackage{enumerate}

\makeatletter \@namedef{subjclassname@2010}{
  \textup{2020} Mathematics Subject Classification}
\makeatother

\newtheorem{thm}{Theorem}[section]

\newtheorem{cor}[thm]{Corollary}
\newtheorem{lem}[thm]{Lemma}
\newtheorem{pro}[thm]{Proposition}
\newtheorem{conj}[thm]{Conjecture}

\theoremstyle{remark}
\newtheorem*{rema}{\textbf{Remark}}

\theoremstyle{definition}

\newtheorem{exa}[thm]{\textbf{Example}}

\newcommand{\Ima}{\operatorname{Im}}

\newcommand{\Real}{\operatorname{Re}}

\newcommand{\mi}{\operatorname{i}}

\newcommand{\R}{\mathbb{R}}

\newcommand{\N}{\mathbb{N}}

\newcommand{\C}{\mathbb{C}}

\begin{document}

\title[Reduction of powers of self-adjoint operators]{On the reduction of powers of self-adjoint operators}
\author[S. Kebli and M. H. Mortad]{Salima Kebli and Mohammed Hichem Mortad$^*$}

\date{}
\thanks{* Corresponding author.}


\address{(Both authors) Department of
Mathematics, University of Oran 1, Ahmed Ben Bella, B.P. 1524, El
Menouar, Oran 31000, Algeria.}

\email{skebli@yahoo.com, kebli.salima@univ-oran1.dz}

\email{mhmortad@gmail.com, mortad.hichem@univ-oran1.dz.}

\dedicatory{}

\begin{abstract}
Let $T\in B(H)$ be such that $T^n$ is self-adjoint for some
$n\in\mathbb{N}$ with $n\geq 3$. The paper's primary aim is to
establish the conditions that lead to the self-adjointness of $T$.
We pay particular attention to the case where $T^3=0$ and how it
implies $T$ is complex symmetric.
\end{abstract}

\maketitle

\section{Introduction}

First, we assume that readers are familiar with definitions and
concepts of bounded linear operators. A suitable reference for this
work is \cite{Mortad-Oper-TH-BOOK-WSPC}. We will recall specific
definitions, however.

In this manuscript, we denote a complex Hilbert space as $H$, which
can be infinite-dimensional. The algebra of bounded linear operators
from $H$ into itself is denoted as $B(H)$. Let $T \in B(H)$. Call
$T$ self-adjoint when $T^* = T$, where $T^*$ is the usual adjoint,
which is the conjugate transpose of $T$ if $T$ is a matrix. We say
that $T$ is normal if $TT^* = T^*T$. Say that $T\in B(H)$ is complex
symmetric if $T=CT^*C$ for some conjugation $C\in B(H)$, where a
conjugation is a conjugate-linear operator $C\in B(H)$ that is both
involutive ($C^2=I$) and isometric (see, e.g., \cite{Garcia-IEOT
T2=0 T complex symmetr}).

Call $T$ a square root of $S$, where $S\in B(H)$, provided $T^2=S$.

The operator $T$ is called positive, denoted as $T \geq 0$, if
$\langle Tx, x \rangle \geq 0$ for all $x \in H$ (in the context of
matrices, this means positive semi-definite). It is well-known that
each positive operator has a unique positive square root. Since
$T^*T$ is positive, it has a unique positive square root, which we
denote by $|T|$, as is customary.

Remember that any operator $T$ in $B(H)$ can be expressed as
$T=A+\mi B$, where $A$ and $B$ are self-adjoint. This decomposition
is called the Cartesian decomposition of $T$. As per convention, we
represent $A$ as $\Real T$ and $B$ as $\Ima T$, where $\Real
T=(T+T^*)/2$ and $\Ima T=(T-T^*)/{2\mi}$.

It is evident that $T$ is self-adjoint if and only if $B=0$. Also,
$T=0$ if and only if $A=B=0$. Furthermore, $T$ is normal if and only
if $AB=BA$.

That a normal operator with a real spectrum is self-adjoint is
typically proven using the spectral theorem. However, a new simple
proof based on the Cartesian decomposition can be found in
\cite{Mortad-INVERTBILITY-SUM}.

Now, consider the equation $T^2=S$, where $S,T\in B(H)$. It is of
some interest to know when $T$ belongs to the same class of $S$. For
example, if $S$ is self-adjoint (respectively normal), when is $T$
self-adjoint (respectively normal)? Just when $H=\C^2$, taking
$S=0$, which is obviously self-adjoint, shows that $T$ can be
anything, e.g., $T=\left(
                                                         \begin{array}{cc}
                                                           0 & 1 \\
                                                           0 & 0 \\
                                                         \end{array}
                                                       \right)$,

which is not even normal (in fact, not quite anything as such a
matrix is necessarily complex symmetric, as in Theorem 5 in
\cite{Garcia-IEOT T2=0 T complex symmetr}).
                                                       On the other
                                                       hand, the
                                                       equation
                                                       $T^2=I$, still in a bi-dimensional space,  can
                                                       have an infinitude of self-adjoint
                                                       solutions.
                                                       See, e.g., Question
                                                       6.2.1 in
                                                       \cite{Mortad-cex-BOOK}.
                                                       Also, the
                                                       shift
                                                       operator on
                                                       $\ell^2$ does
                                                       not have any
                                                       square root.
See Problem 151 in \cite{halmos-book-1982}. In
\cite{Mashreghi-Ptak-Ross-sq rt Studia 2023}, the authors provided
complete descriptions of the set of square roots
of certain classical operators, such as the Volterra and the
Ces\`{a}ro operators, among others.

                                                       Some
                                                       authors have
                                                       sought for
                                                       conditions on
                                                       $T$ that
                                                      place it in
                                                       the same
                                                       class as that
                                                       of $S$. For
                                                       example, C. R. Putnam showed in \cite{Putnam-square-rt-s-a-logarithm} that if
$T\in B(H)$ satisfies $T^2=S$, where $S\geq0$ and $\Real T\geq0$,
then $T$ is necessarily the unique positive square root of $S$. He
also obtained in \cite{Putnam-sq-rt-normal} that if $T^2$ is
self-adjoint and $\Real\langle Tx,x\rangle\neq0$ for all $x\in H$,
then $T$ is self-adjoint. Some other results can be consulted in
\cite{Embry nth roots}, \cite{Kurepa-nth root normal Math. Zeit
1962}, \cite{Mortad-square-root-normal}, \cite{Putnam-sq-rt-normal},
and \cite{Radjavi-Rosenthal-sq-roots-normal}. Special cases
involving possibly unbounded operators are discussed in
\cite{Dehimi-Mortad-squares-polynomials},
\cite{Frid-Mortad-Dehimi-Nilpotence-UNBD} and
\cite{Mortad-JMAA-2023}.

Regarding the class of normal operators, it is also known that if
$T^2$ is normal and $\Real(T) \geq 0$, then $T$ is normal. This
result has appeared in at least two papers in the literature
(\cite{Embry 1966 Normality!!} and \cite{Putnam-sq-rt-normal}). We
may also add \cite{Frid-Mortad-Dehimi-Nilpotence-UNBD}, which
discusses both a specific and a more general version; specific
because $T^2=0$, and more general as it deals with $T^n=0$, where
$n\geq3$. We digress to say that this result already appeared in
\cite{Kurepa-nth root normal Math. Zeit 1962} (Lemma 3) but the
authors of \cite{Frid-Mortad-Dehimi-Nilpotence-UNBD} were not aware
of it, and besides, their proof is different, simpler, and also
dealt with unbounded operators.

There are further results that provide conditions for the operator
$T$ to be in the class of, for example, $T^n$. This type of results
differs from what we are after in this manuscript, as we wanted some
results that make sense even for matrices. For example, some of
these conditions involve non-normal operators such as hyponormal or
quasinormal ones, but these results are not applicable when $\dim H$
is finite because, for example, hyponormality is equivalent to
normality in that case. Other results involving unbounded symmetric
and self-adjoint operators are also straightforward when $T\in
B(H)$, as these concepts coincide in such a case.

In the end, this paper aims to continue the investigation, focusing
on the more general and challenging case of higher powers while
mainly restricting ourselves to self-adjoint operators.

\section{The case of the self-adjointness of $T^3$}

It is plain that if $T\in B(H)$ is self-adjoint, then so are all its
powers $T^n$. Now, if $T^2$ is self-adjoint, then $T^{2p}$ is
self-adjoint for all $p\in\N$, but, e.g., $T^3$ need not be
self-adjoint. Similarly, if $T^3$ is self-adjoint, so are $T^{3q}$
for any $q\in\N$ and $T^2$ could be non-self-adjoint. Therefore, the
following simple observation might be helpful.

\begin{pro}
Let $T\in B(H)$ be such that $T^2$ and $T^3$ are self-adjoint. Then
$T^n$ is self-adjoint for any $n\geq4$, with $n\in\N$.
\end{pro}

\begin{proof}Let $n\geq4$, where $n\in\N$. Write $n=2p+3q$ for some $p,q\in\N$. Since commuting self-adjoint operators are self-adjoint, as are powers of self-adjoint
operators, we have
\[T^n=T^{2p+3q}=T^{2p}T^{3q}=T^{2p}T^{3q}=(T^{2})^p(T^{3})^q,\]
which shows that $T^n$ is self-adjoint.
\end{proof}

\begin{rema}The preceding result for the class of normal operators
first appeared in \cite{Jibril. N power normal operators}, then in
\cite{Alzuraiqi Patel. n normal}. The proof presented here for
self-adjoint operators may also be adapted to the normal ones by
remembering that the product of two commuting normal operators is
normal.
\end{rema}

\begin{rema}
It is worth noting that the previous result and remark also apply to
unbounded operators without any changes or additions.
\end{rema}

From now on, we will be interested in forms of the converse of the
previous result. In the introduction, we have already mentioned that
if $T^2$ is normal while $\Real T\geq0$, then $T$ is normal. This
result is not valid anymore if we replace "normal" with
"self-adjoint" (witness $T=\mi I$). Now, we inquire whether $T$
remains self-adjoint if only $T^3$ is self-adjoint while also
assuming $\Real T\geq0$. The answer is again negative.

\begin{exa}\label{MAIN EXAMPLllessssesssssss EXAmple}
For instance, consider
\[T=\left(
      \begin{array}{cc}
        0 & \mi \\
        \mi & 1 \\
      \end{array}
    \right),
\]
which is defined on $\C^2$. Then $T^3=\left(
      \begin{array}{cc}
        -1 & 0 \\
        0 & -1 \\
      \end{array}
    \right)$, which is self-adjoint, but $T$ itself is not. Observe, in the
    end, that $\Real T=\left(
               \begin{array}{cc}
                 0 & 0 \\
                 0 & 1/2 \\
               \end{array}
             \right)$, which is clearly positive semi-definite.
\end{exa}

In the above example, $T^2=\left(
      \begin{array}{cc}
        -1 & \mi \\
        \mi & 0 \\
      \end{array}
    \right)$ indicating that $T^2$ is not self-adjoint. Therefore, we are inclined to propose the following result:

\begin{thm}\label{T3 T2 self-adj T s.ADJ if re T geq0 THMMmm}
Let $T=A+\mi B$, where $A,B\in B(H)$ are self-adjoint. If $T^3$ and
$T^2$ are self-adjoint with $A\geq0$ (or $A\leq0$), then $T$ is
self-adjoint.
\end{thm}

We present two proofs: the first is based on Cartesian
decomposition, while the second is somewhat simple but relies on a
result by M. R. Embry.

\begin{proof}[First proof of Theorem \ref{T3 T2 self-adj T s.ADJ if re T geq0 THMMmm}] It is easy to see that
\[T^2=A^2-B^2+\mi(AB+BA)\]
and
\[T^3=A^3-B^2A-AB^2-BAB+\mi(A^2B-B^3+ABA+BA^2).\]
Since $A^3-B^2A-AB^2-BAB$ and $A^2B-B^3+ABA+BA^2$ are self-adjoint,
they form the real and imaginary parts of  $T^3$ respectively.

Given that $T^3$ is self-adjoint, its imaginary part must vanish,
leading to $B^3=A^2B+ABA+BA^2$. Furthermore, $T^2$ is self-adjoint;
we must have $AB+BA=0$, or $AB=-BA$. Thus, $A^2B=BA^2$. If $A\geq0$,
then $AB=BA$; and if $A\leq 0$, we obtain $-AB=-BA$. In either case,
$AB=BA$. Consequently, $AB=0$. Returning to the imaginary part of
$T^3$, we see that $B^3=0$, from which we derive $B=0$. Thus, $T$ is
self-adjoint.
\end{proof}

\begin{proof}[Second proof of Theorem \ref{T3 T2 self-adj T s.ADJ if re T geq0 THMMmm}]
Since $T^2$ is self-adjoint, the equation $T^2T=TT^2$ gives
$T^2T^*=T^*T^2$. So
\[T^*TTT^*=T^*T^*TT=TT^*T^*T~(=T^4)\]
(such a $T$ is commonly known as binormal). Thus,
$|T||T^*|=|T^*||T|$. When $\Real T\geq0$, Theorem 2 in \cite{Embry
1966 Normality!!} implies the normality of $T$. By replacing $T$
with $-T$, we observe that the condition $\Real T\leq0$ also ensures
the normality of $T$.

We now prove that $T$ is self-adjoint. Since $T$ is normal, it
suffices to show that its spectrum is a subset of $\R$. Let
$\lambda\in\sigma(T)$. Since both $T^2$ and $T^3$ are self-adjoint,
it follows by the spectral mapping theorem that $\lambda^2$ and
$\lambda^3$ are real, which means that $\lambda$ must also be real.
Therefore, we conclude that $T$ is self-adjoint.
\end{proof}

\begin{rema}
Notice that the self-adjointness of $T^2$ and $T^3$ does not even
yield the normality of $T\in B(H)$. For instance, any nilpotent
non-zero $2\times 2$ matrix is a counterexample.
\end{rema}

\begin{rema}
Example \ref{MAIN EXAMPLllessssesssssss EXAmple} already supplies a
non-normal matrix $T$ such that $T^2$ is not normal either, yet
$T^3$ is self-adjoint and $\Real T\geq0$. This example is somehow
discouraging as we also have $\Real T^2\leq0$. Thus, it seems hard
to obtain the normality of a $T$ if $T^n$ is normal for some
$n\geq3$ and if $\Real T^m\geq0$ (or $\leq0$) for a certain (or all)
$m$.
\end{rema}

A similar reasoning yields the following result:

\begin{thm}
Let $T=A+\mi B$, where $A,B\in B(H)$ are self-adjoint. If $T^3$ and
$T^2$ are self-adjoint with $B\geq0$ (or $B\leq0$), then $T$ is
self-adjoint.
\end{thm}

Thus, we have an intriguing way of characterizing the
self-adjointness of $T$ through its real and imaginary parts, which
conclusively eliminates the possibility of finding superior results.

\begin{cor}\label{T3 s-adj T 2s.a. iff T s.a. Re T posi COR}
Let $T=A+\mi B$, where $A,B\in B(H)$ are self-adjoint. Assume that
$T^3$ is self-adjoint. If $A\geq0$ or $A\leq0$ or $B\geq0$ or
$B\leq0$, then
\[T^2 \text{ is self-adjoint}\Longleftrightarrow T \text{ is
self-adjoint.}\]
\end{cor}

Next, we treat the case of positive operators, which, in light of
Theorem \ref{T3 T2 self-adj T s.ADJ if re T geq0 THMMmm}, has now
become straightforward to deal with.

\begin{cor}\label{T3 T2 pos T s.ADJ if re T geq0 cor}
Let $T=A+\mi B$, where $A,B\in B(H)$ are self-adjoint. If $T^3$ and
$T^2$ are positive with $B\geq0$ (or $B\leq0$), then $T$ is
positive.
\end{cor}

\begin{proof}
Since positivity is stronger than self-adjointness, by Theorem
\ref{T3 T2 self-adj T s.ADJ if re T geq0 THMMmm}, we can conclude
that $T$ is self-adjoint. Let $\lambda$ be an a priori complex
number in  $\sigma(T)$. Since $T^2,T^3\geq0$, we can infer that
$\lambda^2,\lambda^3\geq0$. Since the condition $\lambda^2\geq0$
implies the realness of $\lambda$, $\lambda^3\geq0$ yields
$\lambda\geq0$, which means that $T$ is positive, marking the end of
the proof.
\end{proof}

\begin{rema}
If we assume that $\Real T\geq0$, then $T$ is positive if $T^2$ is
positive, that is, without needing to add the positivity of $T^3$ (a
similar conclusion is reached if we suppose $\Real T\leq0$ in lieu).
We already alluded that this result appeared in
\cite{Putnam-square-rt-s-a-logarithm}. To show that the above
corollary is not covered by the result of
\cite{Putnam-square-rt-s-a-logarithm}, let $T=-I$. Then $T^2\geq0$
and $\Ima T=0$, yet $T$ is not positive, hence the need for assuming
$T^3\geq0$.
\end{rema}

If the positivity of the real or imaginary part of $T$ is dropped,
then $T$ remains self-adjoint but under a specific assumption. For
instance, if $T\in B(H)$ is invertible and $T^2$ and $T^3$ are
self-adjoint, then $T$ is also self-adjoint. This was demonstrated
in \cite{Dehimi-Mortad-squares-polynomials} for two relatively prime
numbers (which need not be just 2 and 3) and for an unbounded
self-adjoint $T$. The subsequent result represents a slight
improvement of this finding.

\begin{pro}
Let $T\in B(H)$ be such that $T^2$ and $T^3$ are self-adjoint. If
$\ker T=\ker T^2$ or $\ker T^*=\ker T^2$, then $T$ is self-adjoint.
\end{pro}

\begin{proof}Since $T^2$  and $T^3$ are self-adjoint, we have
$T^2T^*=T^*T^2$ and $T^3T^*=T^*T^3$. Thus,
\[T^3T^*=T^*T^3=T^*T^2T=T^2T^*T,\]
that is, $T^2(TT^*-T^*T)=0$. Since $\ker T=\ker T^2$, we even have
$T(TT^*-T^*T)=0$, or equivalently, $TTT^*=TT^*T$. This says that
$T^*$ is quasinormal. But a hyponormal operator, which is a weaker
notion than quasinormality, is self-adjoint as soon as it has a real
spectrum (see \cite{Stampfli hyponormal 1965}). Arguing as in the
second proof of Theorem \ref{T3 T2 self-adj T s.ADJ if re T geq0
THMMmm}, we can establish the self-adjointness of $T^*$ or $T$.

In case $\ker T^*=\ker T^2$, we have
\[T^2(TT^*-T^*T)=0 \Longrightarrow T^*(TT^*-T^*T)=0,\] which means that $T$ is
quasinormal. Thus, and as above, $T$ is self-adjoint.
\end{proof}

What other generalizations of Theorem \ref{T3 T2 self-adj T s.ADJ if
re T geq0 THMMmm} are possible? One could conjecture the following:

\begin{conj}If the real part of $T\in B(H)$ is positive and both $T^3$ and $T^4$ are
self-adjoint, then $T$ is self-adjoint.
\end{conj}

Before attempting to address this conjecture, however, readers
should first consider the following example:

\begin{exa}
Let
\[T=\left(
      \begin{array}{cccc}
        0 & 1 & 0 & 0 \\
        0 & 0 & 1 & 0 \\
        0 & 0 & 0 & 0 \\
        0 & 0 & 0 & 1 \\
      \end{array}
    \right),
\]
which is not self-adjoint. Then $T^2$ is not self-adjoint because
\[T^2=\left(
      \begin{array}{cccc}
        0 & 1 & 0 & 0 \\
        0 & 0 & 0 & 0 \\
        0 & 0 & 0 & 0 \\
        0 & 0 & 0 & 1 \\
      \end{array}
    \right).\]
    However,
\[T^3=T^4=\left(
      \begin{array}{cccc}
        0 & 0 & 0 & 0 \\
        0 & 0 &0 & 0 \\
        0 & 0 & 0 & 0 \\
        0 & 0 & 0 & 1 \\
      \end{array}
    \right)\]
    are self-adjoint (and idempotent as well). Observe in the end that the self-adjoint $\Real T$ is not positive (or negative) semi-definite, as it has eigenvalues of opposite signs,
    namely $\frac{\sqrt{2}}{2}$ and $-\frac{\sqrt{2}}{2}$.
\end{exa}

Therefore, this example does not answer the conjecture above, but it
could be useful for other purposes. Still related to the above
conjecture, we have the following observation:

\begin{pro}
Let $T\in B(H)$ be such that $T^3$ and $T^4$ are self-adjoint. If
$\Real T\geq 0$ and $\Real(T^2)\geq 0$, then $T$ is self-adjoint.
\end{pro}

\begin{proof}
Since $T^3$ is self-adjoint, so is $T^6$. Thus, by Corollary \ref{T3
s-adj T 2s.a. iff T s.a. Re T posi COR} and the self-adjointness of
$T^4$, $T^2$ too is self-adjoint as $\Real(T^2)\geq0$. Since $\Real
T\geq0$, the self-adjointness of both $T^2$ and $T^3$ yields that of
$T$ using the same corollary.
\end{proof}

\begin{rema}
The conclusion of the previous result stays unchanged when $\Ima T$
(or $\Ima T^2$) substitutes $\Real T$ (or $\Real T^2$) and when
"$\geq0$" is exchanged with "$\leq0$".
\end{rema}

\section{The general case of an odd power}

\begin{thm}\label{hhichhemmmmmmmmmmmmmm THMMMMmmmm}
Let $T=A+\mi B$, where $A,B\in B(H)$ are self-adjoint. Let $n\in\N$
be such that $n\geq 2$. Suppose $T^{2n+1}$ and $T^2$ are
self-adjoint with $A\geq0$. Then $T$ is self-adjoint.
\end{thm}

\begin{proof}Since $T^2$ is self-adjoint, we have $AB+BA=0$. As shown in the first proof of Theorem \ref{T3 T2 self-adj T s.ADJ if re T geq0 THMMmm}, we can deduce
that $BA=AB=0$. To find the imaginary part of
$T^{2n+1}$, which is nil, there is no need to expand $T^{2n+1}$
using the Binomial theorem. It is evident that each expression of
the form $A^pB^q$, for any $p,q\in\N$, is zero due to the condition
$AB=BA=0$. Thus, we are left with $B^{2n+1}=0$, which then leads to
$B=0$, making $T$ self-adjoint, as wished.
\end{proof}

\begin{rema}
It is primordial to have an odd power in the preceding theorem.
Indeed, if $T=\left(
                \begin{array}{cc}
                  \mi & 0 \\
                  0 & -\mi \\
                \end{array}
              \right)$. Then $T^2=-I$ and $T^4=I$ are self-adjoint,
              $\Real T=0$, and yet $T$ is not self-adjoint. What
              went wrong is the fact that $T^3=\left(
                \begin{array}{cc}
                  -\mi & 0 \\
                  0 & \mi \\
                \end{array}
              \right)$ is not self-adjoint.
\end{rema}

\begin{rema}
Here, too, and as in Corollary \ref{T3 s-adj T 2s.a. iff T s.a. Re T
posi COR}, if $T^{2n+1}$ is self-adjoint and $\Real T\geq0$, then
$T$ is self-adjoint if and only if $T^2$ is self-adjoint.
\end{rema}

\section{The case of nilpotent operators}

If $T\in B(H)$ is such that $T^3=0$, then there is, a priori, no
reason why we should have $T^2=0$, unless $\dim H=2$. The following
result provides a way of reducing the index of nilpotence for an
operator on an infinite-dimensional space.

\begin{thm}\label{T3=0 gives T2=0 if re T2 positive THmmMM}
Let $T\in B(H)$ be such that $T^3=0$, where $\dim H\geq3$. If
$\Real(T^2)\geq0$ or $\Ima (T^2)\geq0$, then $T^2=0$.
\end{thm}

\begin{rema}
Readers will notice from the proof of the previous result that we
will obtain the same outcome if we substitute "$\geq 0$" with "$\leq
0$".
\end{rema}

The above result could be shown using various methods. We have
chosen one based on the Cartesian decomposition. First, we require
the following auxiliary result.

\begin{lem}\label{RS=-SR S2=R2 S R Hermitian S=R=0 Lemmmmmmma}
Let $R,S\in B(H)$ be self-adjoint, with one being positive. Assume
$S^2=R^2$ and $SR=-RS$. Then it follows that $R=S=0$.
\end{lem}

\begin{proof}
Let us assume $R\geq0$ without loss of generality. Since $SR=-RS$,
we can deduce that $SR^2=R^2S$, and therefore $SR=RS$. Thus, $SR=0$,
which implies $SR^2=S^2R=0$. As a result, $R^3=S^3=0$, and
consequently, $R=S=0$, as desired.
\end{proof}

\begin{rema}There are self-adjoint non-zero matrices $R$ and $S$ that
satisfy $S^2=R^2$ and $SR=-RS$. Indeed, consider two of the Pauli
spin matrices on $\C^2$, namely:
\[R=\left(
      \begin{array}{cc}
        0 & 1 \\
        1 & 0 \\
      \end{array}
    \right)\text{ and }S=\left(
      \begin{array}{cc}
        0 & -\mi \\
        \mi & 0 \\
      \end{array}
    \right)
\]
(more commonly denoted by $\sigma_x$ and $\sigma_y$). It is clear
that $RS=-SR$ and $R^2=S^2$. However, neither matrices $R$ and $S$
are positive semi-definite. More generally, if $S$ and $R$ are two
self-adjoint operators and one of them is positive, then $SR=\lambda
RS\neq0$, where $\lambda\in\C$, implies that $\lambda=1$ only. This
result is originally from \cite{BBP}, and a different proof can be
found in \cite{Chellali-Mortad-2014}.
\end{rema}

\begin{proof}[Proof of Theorem \ref{T3=0 gives T2=0 if re T2 positive
THmmMM}] Write $T=A+\mi B$, where $A,B\in B(H)$ are self-adjoint. We
have
\[T^3=A^3-B^2A-AB^2-BAB+\mi(A^2B-B^3+ABA+BA^2).\]
Since $A^3-B^2A-AB^2-BAB$ and $A^2B-B^3+ABA+BA^2$ are self-adjoint,
and $T^3=0$, it ensues that
\[A^3-B^2A-AB^2-BAB=0\text{ and }A^2B-B^3+ABA+BA^2=0.\]
These last two equations may be rewritten in different forms, e.g.,
\[(A^2-B^2)A=(AB+BA)B,~(B^2-A^2)B=(AB+BA)A,\]
\[A(A^2-B^2)=B(AB+BA),~B(B^2-A^2)=A(AB+BA).\]
Therefore,
\[(AB+BA)A^2=(B^2-A^2)BA\text{ and }(A^2-B^2)AB=(AB+BA)B^2,\]
and so
\[(AB+BA)(A^2-B^2)=-(A^2-B^2)(AB+BA).\]
On the other hand,
\[(A^2-B^2)A^2=(AB+BA)BA\text{ and }(B^2-A^2)B^2=(AB+BA)AB,\]
thereby
\[(A^2-B^2)^2=(AB+BA)^2.\]

Since $A^2-B^2$ and $AB+BA$ are both self-adjoint and,
$A^2-B^2\geq0$ (or $AB+BA\geq0$), Lemma \ref{RS=-SR S2=R2 S R
Hermitian S=R=0 Lemmmmmmma} implies that $A^2-B^2=AB+BA=0$. Since
$T^2=A^2-B^2+\mi(AB+BA)$, it is seen that $T^2=0$, as required.
\end{proof}

\begin{cor}
Let $T\in B(H)$ be such that $T^3=0$, where $\dim H\geq3$. If
$\Real(T^2)\geq0$ or $\Ima (T^2)\geq0$, then $T$ is complex
symmetric.
\end{cor}

\begin{proof}
Based on Theorem \ref{T3=0 gives T2=0 if re T2 positive THmmMM}, we
have $T^2=0$, and according to Theorem 5 in \cite{Garcia-IEOT T2=0 T
complex symmetr}, it follows that $T$ is complex symmetric.
\end{proof}

\begin{rema}
The preceding corollary is interesting in the sense that the authors
of \cite{Garcia-Wogen-TAMS complex symtric T2=0 normal T complex
symmetric} showed in their Theorem 2 that for every finite $n\geq3$
and for every $H$ with $\dim H\geq n$, there is an algebraic
operator $T$ on $H$, that is, $p(T)=0$ for some polynomial $p$, of
degree $n$, which is not a complex symmetric operator.
\end{rema}

We conclude with the following simple observation.

\begin{pro}
Let $T\in B(H)$ be such that $T^n=0$ for some $n\geq2$, where
$n\in\N$. Then $T^{n-1}$ is self-adjoint if and only if $T^{n-1}=0$.
\end{pro}

\begin{proof}We only show that the self-adjointness of $T^{n-1}$ implies
$T^{n-1}=0$. Since $T^n=0$, it follows that
$(T^{n-1})^2=T^{2n-2}=T^nT^{n-2}=0$, which yields $T^{n-1}=0$, as
wished.
\end{proof}

\bibliographystyle{amsplain}

\end{document}